\documentclass{svjour2}                    
\smartqed  
\usepackage{amsmath}
\usepackage{amsfonts,color}
\usepackage[
  colorlinks=true,
  bookmarks=true,
  bookmarksnumbered=false,
  bookmarkstype=toc]{hyperref}

\begin{document}

\spdefaulttheorem{myclaim}{Claim}{\bf}{\rm}
\newtheorem{fact}{Fact}

\newtheorem{assumption}{Condition}

\newenvironment{potthm}{\textit{Proof} of Theorem~\ref{thm:2}}{\qed}
\newenvironment{potclaim}{\textit{Proof} of Claim~\ref{claim:1}}{\qed}

\newcommand{\set}[1]{\left\{#1\right\}}
\newcommand{\ssep}{\mathrel{;}}

\def\rvec#1#2{(#1)_{#2}}
\def\cb#1{C_c(#1)}
\def\cbp{C_c(\Gint)}

\def\E#1{\mathrm{E}[#1]}
\def\V#1{\mathrm{V}[#1]}

\def\Cset{\mathbb{C}}
\def\Nset{\mathbb{N}}

\def\Rset{\mathbb{R}}

\def\Tset{\mathbb{T}}
\def\Zset{\mathbb{Z}}

\def\O{\mathcal{O}}
\def\S{\mathcal{S}}

\def\L{\mathcal{L}}
\newcommand{\cL}{\mathcal{L}}

\def\supp#1{\mathop{\mathrm{supp}}{#1}}

\def\cconj#1{{\overline{#1}}}

\def\fG{{\mathfrak{G}}}
\def\fT{\mathfrak{T}}

\newcommand{\rint}{\mathrm{int}}

\def\Gint{{{\mathfrak{G}}_\rint}}
\def\Hint{{{\mathfrak{H}}_\rint}}

\newcommand{\dual}{\widehat}
\newcommand{\pdual}{\widehat{\ }\,}

\def\dualG{{\dual{\fG}}}
\def\dualTset{{\dual{\fT}}}

\def\XP0{{X_{\Lambda}}}
\def\XPa{{X'_{\Lambda}}}

\def\inner#1#2{\langle#1,\,#2\rangle}

\def\cov{\mathrm{Cov}}

\def\fra#1#2{{#1}/{#2}}

\title{Random fields on model sets with localized dependency and their diffraction
\thanks{The preliminary version is presented at the fifth Asian
International Workshop on Quasicrystals, 2009, Tokyo Japan.}
}


\author{Yohji Akama         \and
        Shinji Iizuka
}


\institute{Yohji Akama\at
              Mathematical Institute \\
              Tel.: +81-(0)22-795-7708\\
              Fax: +81-(0)22-795-6400\\
              \email{akama@math.tohoku.ac.jp}           
           \and
           Shinji Iizuka\at
Research and Development Section, Hitachi East Japan
 Solutions, Ltd.
}

\date{Received: date / Accepted: date}

\maketitle

\begin{abstract}
For a random field on a general discrete set, we introduce a condition
that the range of the correlation from each site is within a
predefined compact set $D$. For such a
random field $\omega$ defined on the model set $\Lambda$ that satisfies a natural geometric condition, we develop a method to
calculate the diffraction measure of the random field. The method
partitions the random field into a finite number of random fields, each being
independent and admitting the law of large numbers. The diffraction measure 
 of $\omega$ consists almost surely of a pure-point component and an
 absolutely continuous component. The former is the diffraction
 measure of the expectation $\E{\omega}$,
while the inverse Fourier transform of the absolutely
 continuous component of $\omega$ turns out to be a weighted Dirac comb
 which satisfies a simple formula. Moreover, the pure-point component will be
 understood quantitatively in a simple exact  formula if the weights are continuous over the internal
 space of $\Lambda$. Then we provide a sufficient condition that the diffraction
 measure of a random field on a model set is still pure-point.

\end{abstract}

\section{Introduction}

A physical quasicrystal is a material which has (1) a diffraction pattern with
Bragg peaks and (2) a symmetry that ordinary crystals cannot have. The set of the atomic positions in
a quasicrystal is mathematically modelled by a \emph{model set}~\cite{M:meyer}, which
is defined by introducing an extra space (``internal space''),  a relatively
compact subset (``\emph{window}'') of the internal space, and a so-called
\emph{star map} $(-)^\star$ such that for each site $s$ of a model set
the point $s^\star$ belongs to the window. The topological
properties of the window cause the
pure-point diffraction measure (see \cite{MR1798991} and
\cite{MR2084582} for example) of the model
set, which explain the aforementioned properties (1)
 of the quasicrystal.

Although model sets are proved to have necessarily a pure-point diffraction measure, real quasicrystals have
diffraction measures with not only Bragg peaks (pure-point component) but
also diffuse scattering (absolutely continuous component).  The
phenomenon is explained from a physical point of view with a probabilistic effect in
\cite{janot94:_quasic}, and in
\cite{baake98:_diffr_point_sets_with_entrop} where to the sites of the model set associated are
\emph{independent} random variables. In \cite{MR1325836}, Hof regarded
 the thermal motion  of atoms as i.\@i.\@d.\@ random displacements, and
 then studied the influence on the diffraction measure of aperiodic
 monoatomic crystals.
Since correlations (\cite{barabash09:_diffus_scatt_and_fundam_proper_of_mater}) are present
in a quasicrystal, we equip model sets with a localized probabilistic dependency,
 to quantitatively study the ability of diffuse
scattering to characterize local structures and defects in materials.
In \cite{MR2480747}, Lenz employed a dynamical system of point sets to
study the diffraction measures of percolation and the random
displacement  models based on aperiodic order. Recently, in \cite{MR:2011}, M\"uller and
Richard also made a rigorous approach on these models by using sets of $\sigma$-algebras.

 For a model set $\Lambda$, we consider a complex-valued random field $\{X_s\}_{s\in\Lambda}$ with 
dependency localized as follows:
there is a
finite patch $D$ such that each site $s$ has correlation on, at
most, sites belonging to the patch $D$ relative to $s$. 
This localized dependency condition seems
essentially the same as the so-called ``finite range condition'' of stochastic analysis.
We call a random field on a model set subject to the localized
dependency condition a \emph{finitely randomized model set}. We develop a method to
calculate the diffraction measure of such complex-valued random field
(Section~\ref{sec:almostisolated}). 
For the diffraction
measures of finitely randomized model sets,
we determine quantitatively
the pure-point component (in Section~\ref{sec:intensity}) and the
absolutely continuous component (in Section~\ref{sec:debye}). Our approach is mostly based on the
\emph{finite local complexity} of model sets, as in  \cite{MR1798991}. As a
consequence,
if the fourth noncentral moments $\{\E{\,|X_s|^4}\,\}_{s\in\Lambda}$ is bounded, and
if the expectation $\E{X_s}$ at each site $s$ as well as the covariance
$X_s$ and $X_{s-g}$ of sites $s$ and $s-g$
 in the finitely randomized model set are given by  continuous
 functions $e(s^\star)$ and $c_g(s^\star)$, where $s^\star$ is  the value of $s$ by the star map, 
then 
\begin{enumerate}
\item the inverse Fourier transform of the
 absolutely continuous component is a Dirac comb such that
the support $\{g_1,\ldots, g_n\}$ is the
 smallest $D$ and the weight of the $\delta_{g_i}$ is the
      average  of the covariances $\cov[X_s,\ X_{s-g_i}]$ ($s, s-g_i\in
      \Lambda$); and

\item the pure-point component 
 is the diffraction measure of a Dirac comb $\sum_{s\in\Lambda}
      \E{X_s}\delta_s$, the  expectation of the random field.
\end{enumerate} 
This type of theorems are also seen in some other models with i.i.d. conditions. See
\cite{baake10:_diffr_of_stoch_point_sets} and references therein.

On the other hand, 
from the viewpoint of stochastic processes, we provide a
sufficient condition for a
randomly weighted Dirac comb on a
model set to have
diffraction measure whose expectation is still pure-point.
 The sufficient condition is satisfied when the set of the weights $X_s$
 forms a Wiener stochastic
process  $\{X_{\star^{-1}(y)}\}_{y\in W}$ parametrized by the window
$W \subset \Rset$. We draw this observation by providing
quantitatively the diffraction measure of the deterministically weighted
Dirac comb on a model set. The quantitative estimate will be
established with the help of the so-called torus parametrization which
was introduced in \cite{0305-4470-30-9-016}, then was extended in
\cite{MR1798991} to the locally compact, $\sigma$-compact Abelian
Hausdorff groups (\textsc{lcag} for short), and was finally fully
exploited in \cite{MR2308136}.

\section{Basic properties of model sets: review}

Throughout this paper, $\fG$ and $\Gint$ are locally compact,
 $\sigma$-compact Abelian Hausdorff groups (\textsc{lcag} for
 short). 

\begin{definition}\label{def:ms}
A \emph{cut-and-project scheme} (c.-p.\@ scheme, for short) is a triple $\S=(\fG,\Gint,\tilde L)$
such that
(1) $\fG$ and $\Gint$ are called a \emph{physical space} and an \emph{internal
 space} respectively;
(2) $\tilde L$ is a lattice of $\fG\times\Gint$, that is,
 a discrete
 subgroup of $\fG\times \Gint$ with $(\fG\times\Gint)/\tilde L$ being
 compact;
(3) The canonical projection $\Pi:\fG\times\Gint\to\fG$ is injective on
 $\tilde L$, and the image of $\tilde L$ by the other canonical
 projection $\Pi_{\mathrm{int}}:\fG\times\Gint\to\Gint$  is dense in the internal space $\Gint$.
For each $s\in \Pi(\tilde L)$, we write $s^\star$ for $\Pi_{\mathrm{int}}\circ
 (\Pi|_{\tilde L})^{-1}(s)$, $L$ for $\Pi(\tilde L)$, and $L^\star$ for
 $(L)^\star$.  The $(-)^\star$ is called the \emph{star map}.
Define
\begin{align*}
\Lambda_\S(W):=\{ \Pi(x)\;;\; x\in\tilde L,\ \Pi_{\mathrm{int}}(x)\in W\} .
\end{align*}
\end{definition}
For any \textsc{lcag} $G$, any subsets $A,B\subset G$, and any $x\in G$, let $A\pm B$ be $\{a\pm b\;;\; a\in
A, b\in B\}$ and $x+A$ be $\{x+a\;;\;a\in A\}$. A set $\Lambda\subset\fG$ is said to be
\emph{uniformly discrete}, if $(\Lambda-\Lambda)\cap U=\{0\}$ for some
open neighborhood $U$ of 0, while $\Lambda$ is said to be
\emph{relatively dense}, if $\fG=\Lambda + K$ for some compact set $K$.
The interior, the closure and the boundary of $A$ in
\textsc{lcag} $G$ are denoted by  $Int_G(A)$, $Cl_G(A)$ and $\partial_G
A$, respectively. Here the subscript $G$ is omitted if no confusion occurs.

Lattices in $G\times H$ with $G,H$ being \textsc{lcag}'s are written as
 $\tilde L,\tilde M,\ldots$,
 and their canonical projections to $G$ are written as $L,M,\ldots$,
 respectively.

\begin{lemma}\label{lem:ud}
If $\Lambda\subset\fG$ is uniformly discrete, there is a compact neighborhood
 $U$ of 0 such that $U=-U$ and $(s+U)\cap (s'+U)=\emptyset$ for
 all distinct $s,s'\in \Lambda$.\end{lemma}


\begin{definition}[Model set] Let $(\fG, \Gint, \tilde
 L)$ be a c.-p.\@ scheme. By a \emph{window}, we mean a
 nonempty, measurable, relatively compact subset of the
 internal space $\Gint$.  If $W\subset\Gint$ is a window, $\Lambda(W)$ is called
a \emph{model set}.\end{definition}

It is well-known that any model set is
uniformly discrete.  See \cite[Proposition~2]{MR1884143}, for example.
Every \textsc{lcag} has a unique Haar measure up to positive real multiples.
Throughout this paper, we fix Haar measures of the
\textsc{lcag}'s $\fG$ and $\Gint$. The Haar measure of $\Gint$ is denoted
by $\theta$,
 and the integration of a function with
respect to the Haar measure of $\fG$ ($\Gint$ resp.) is denoted by
$\int\cdots d x$ ($\int\cdots d y$ resp.) as usual. The Haar measure of a
set $A$ is just denoted by $|A|$ if no confusion occurs.
 By a \emph{van Hove sequence} of
$\fG$, we mean an increasing sequence $\{D_n\}_{n\in\Nset}$ of compact subsets of
$\fG$ such that $|D_n|>0$  for every $n\in\Nset=\{1,2,3,\ldots\}$ and for every  compact subset
$K\subset\fG$, $\lim_{n\to\infty} \left|\partial^K(D_n)\right|/|D_n|=0$, where for a compact set $A\subset\fG$,
\begin{align}
\partial^K(A):= \bigl((A+K)\setminus Int(A)\bigr) \cup \Bigl(\bigl(
 Cl(\fG\setminus A)-K\bigr)\cap A\Bigr). \label{thick boundary}
\end{align}
In \cite{MR1798991}, the existence of a van Hove sequence for any \textsc{lcag} is derived from 
\begin{proposition}[\protect{\cite[Theorem~9.8]{MR551496}}]\label{thm:str}
For every
locally compact, compactly generated Abelian Hausdorff 
group $H$, there are $l,m\in\Zset_{\ge0}$, a compact Abelian Hausdorff
 group $K$, and an isomorphism $\varphi$ from $H$ to 
$\Rset^l\times\Zset^m\times K$.
\end{proposition}

Below we fix a van Hove sequence $\{D_n\}_n$ of $\fG$. For a lattice $\tilde L\subset\fG\times\Gint$,
 the measure of the fundamental domain is denoted by $|\tilde L|$, where
 the measure is the product measure of the Haar measures of $\fG$ and
 $\Gint$.

\def\dense{\mathrm{dense}}
For each discrete set $\Lambda$ of $\fG$, the density of $\Lambda$ with
respect to the van Hove sequence $\{D_n\}_n$ is denoted by 
$\dense_{\{D_n\}_n}(\Lambda):=\lim_{n\to\infty} \sum_{s\in \Lambda\cap
 D_n} |D_n|^{-1}$.

 The set of continuous functions from a set $A$ to a  set $B$ is denoted by  $C(A,B)$. 
 Let $1_V:\Gint\to\{0,1\}$ be the
indicator function of a set $V$, that is,
$1_V(y)= 1$ for $y\in V$, and $0$ otherwise.

We recall \cite[Proposition~6.2]{MR1798987} as follows
(see also \cite{MR1884143}):

\begin{proposition}[Weyl's theorem for model sets]\label{prop:weyl}
For any van Hove sequence $\{D_n\}_{n\in\Nset}$ of $\Gint$ and for
 any relatively compact set
 $W\subset\Gint$ with the Haar measure $\theta(\partial W)$ being $0$ and for any 
 $f\in \cbp$, we have
\begin{align*}
 \lim_{n\to\infty} \frac{1}{\theta(D_n)}{\sum_{s\in\Lambda(W)\cap D_n}
 f(s^\star)} = \frac{1}{|\tilde
 L|} \int_W f(y) d  y.\end{align*}
\end{proposition}

\subsection{Diffraction}  Mathematical diffraction theory, introduced
 by Hof~\cite{MR1328260}, is reviewed below
according to \cite{MR2084582}. We say a countable set $S\subset\fG$ has
\emph{finite local complexity}~(\textsc{flc}), if $S-S$ is closed and discrete. Let a bounded complex sequence
$\bigl\{w_s\bigr\}_{s\in S}$ be indexed over an \textsc{flc} set $S\subset\fG$ such that the corresponding Dirac comb $
 \omega:=\sum_{s\in S} w_s \delta_s$
defines a regular Borel measure on $\fG$. Here $\delta_s$ is a Dirac
measure of $\fG$ such that $\delta_s(A)=1$ if $s\in A$ and 0 otherwise
for any $A\subset \fG$.  We often identify $\omega$
with $\{w_s\}_{s\in S}$.  For a van Hove
sequence  $\{D_n\}_n$ of $\fG$,
set $\omega_n:=\omega|_{D_n}:=\sum_{s\in S\cap D_n}w_s\delta_s$.
For any complex measure $\mu$ on $\fG$, let $\tilde \mu$ be $\tilde
{\mu}(A)=\cconj{\mu(-A)}$. 
Set
$\gamma_\omega^{(n)}:={\omega_n\ast \widetilde{\omega_n}}/{|D_n|} 
$, where $*$ is the convolution. Actually we have
$ \gamma_\omega^{(n)}=\sum_{g\in S- S}
\eta^{(n)}_\omega(g) \delta_g$ where
$\eta^{(n)}_\omega(g)$ is the summation of $w_s
\cconj{w_t}|D_n|^{-1}$ over $s,t\in S\cap
  D_n$ such that $  s-t=g$. Since $S$ is discrete and $D_n$ is compact,
  $\eta^{(n)}_\omega$ is well-defined, and $\gamma^{(n)}_\omega$ is so.

Let $C_c(\fG)$ be the set of complex continuous functions on $\fG$ with
compact support in $\fG$.
Let the \emph{autocorrelation measure} of $\omega$ be the limit $\gamma_\omega$ of $\gamma_\omega^{(n)}$ in the
vague topology. Then $\gamma_\omega$ is written as 
\begin{align*}
\gamma_\omega=\sum_{g\in S - S}
 \eta_\omega(g)\delta_g,\qquad \eta_\omega(g)=\lim_{n\to\infty}\eta_\omega^{(n)}(g).
\end{align*} 
For any \textsc{lcag} $H$, the \emph{dual
group} of $H$ is denoted by $\dual{H}$. Fix a Haar measure $h$ of $H$. 
The Fourier
transform of a function $f:H\to \Cset$ is $\hat{f}:\dual{H}\to
\Cset$ such that $\hat{f}(\chi):=\int_H
f(x)\cconj{\chi(x)}d h(x)$. The \emph{diffraction measure} of $\omega$
 is, by definition, the Fourier
transform 
$\widehat{\gamma_\omega}$ of the autocorrelation measure $\gamma_\omega$. 
A measure on $\fG$ has Fourier transform as a measure on $\widehat \fG$ as
follows:

\begin{proposition}[\protect{\cite[Theorem~2.1, Theorem~4.1]{MR0621876}}]\label{prop:al}
Suppose $\lambda$ is a measure on $\fG$. 
The Fourier transform (if it exists) of $\lambda$ is a \emph{unique} measure $\widehat{\lambda}$ on $\dual\fG$ such that
for all $\varphi \in C_c(\fG)$ 
\begin{align*}\int_\fG (\varphi  * \tilde{\varphi })(x) d\lambda(x) =
\int_{\dual \fG} \left| \widehat {\varphi }( - \chi)\right|^2
d\widehat{\lambda}(\chi).
\end{align*}  
Moreover if $\lambda$ is positive definite (i.e.,
 $\int_\fG (\varphi * \tilde{\varphi})(x) d\lambda(x)\ge0$ for all
 $\varphi\in C_c(\fG)$), then
 $\lambda$ indeed has the Fourier transform $\widehat\lambda$ which is positive.
Here
 $\tilde \varphi:\fG\to\Cset$  is a function 
$x\in\fG\mapsto \cconj{\varphi(-x)}$.
\end{proposition}

A Haar measure on $\dual \fG$ is $\widehat\delta_0$ where $\delta_0$
is the Dirac measure at $0$ on $\Gint$. Then the equation of
Proposition~\ref{prop:al} amounts to a Plancherel formula. The integral of
a function with respect to the Haar measure of $\dual \fG$ is denoted by
$\int \cdots d\chi$. Because $\widetilde{\eta_\omega}=\eta_\omega$, we
have $\widetilde{\gamma_\omega}=\gamma_\omega$ and thus
$\widehat{\gamma_\omega}$ is positive by \cite[Proposition~4.1]{MR0621876}.

A measure $\mu$ on $\fG$ is said to be \emph{translationally bounded}, if for
every $\varphi\in C_c(\fG)$ the set $\left\{ \int_\fG \varphi(x + t)d\mu(x)\;; \;
t\in \fG\right\}$ is bounded.

Baake-Moody established the pure-point diffraction of weighted Dirac
combs on model sets, by using Weyl's theorem (\cite[Proposition~6.2]{MR1798987})
for model sets  and an ingenious topological space.
\begin{proposition}[\protect{\cite[Theorem
 2]{MR2084582}}]\label{prop:modelsetpp}Suppose that $\Lambda(W)$ is a
 model set, and
 the window $W\subset\Gint$  
satisfies that the Haar measure $\theta(\partial W)$ is 0. Then for
 any
 $f:\Gint\to\Cset$ supported and continuous on $Cl(W)$,
 the diffraction measure $\dual{\gamma_\omega}$ of the Dirac
 comb $\omega=\sum_{s\in\Lambda(W)} f(s^\star)\delta_s$
is
translationally bounded, nonnegative and pure-point. \end{proposition}

\section{Finitely randomized model sets}\label{sec:almostisolated}

We need Weyl's theorem for random fields on model sets, in
Section~\ref{sec:debye}. So for a model set $\Lambda_\O(W)$ over a
c.-p.\@ scheme $\O$, we wish to decompose a random field $\{X_s\}_{s\in
\Lambda_\O(W)}$ into a finite number of \emph{independent} random
fields, each being on a model set over the same c.-p.\@ scheme $\O$.
For a sufficient condition of such a decomposition, we consider a
geometric condition~(Subsection~\ref{subsec:finitely periodic}) on model
sets and a condition for random fields~(Subsection~\ref{subsec:depset}).
Hereafter, a ``random variable'' is abbreviated as ``\textsc{rv},'' and
the cardinality of a finite set $A$ is denoted by $\#A$.

\subsection{Independence in random field}\label{subsec:depset}

\begin{definition}[Dependency set] \label{def:depset} Let
 $\{X_s\}_{s\in S}$  be  a random
 field on a discrete set $S$. 
A \emph{dependency set}(\emph{d-set} for short) is a set $D=-D\subset S-S$ such that 
for any finite sets $P, Q \subset S$, 
if a set $(P - Q)$ is disjoint from $D$, then a pair of a $\#P$-dimensional random vector
 $\rvec{X_s}{s \in  P}$ and a $\#Q$-dimensional random vector $\rvec{X_s}{s
      \in  Q}$ is independent. A d-set
 necessarily has $0$ as an element. If a random field has a dependency
 set, we can replace it with an arbitrary superset of it.
\end{definition}

A d-set is a patch such that
 each site $s$ has correlation on, at most, sites
 belonging to the patch relative to $s$.  
 Recall that a sequence $\{X_1,X_2,\ldots\}$ of \textsc{rv}'s is
 independent, if so are any finite subsequences. We set the product
 $\Pi_{x\in \emptyset}(\cdots)$ to be 1.

\begin{lemma}[Independence]\label{lem:irms} Let $D$ be a d-set
 of a random field $\{X_s\}_{s\in S}$ on an \textsc{flc} subset $S$ of an
 \textsc{lcag}, and let $N$ be any subset of the \textsc{lcag}. If $D\cap \left( (s+N) -
 (t+N)\right)=\emptyset$ for any distinct $s, t \in S $, then a sequence
 $\left\{\; \prod_{t\in (s+N)\cap S}X_t\;;\;s\in S\right\}$ of
 \textsc{rv}'s is
 independent. Furthermore, the random field is independent, if and only
 if the random field has $\{0\}$ as a d-set.
\end{lemma}
\begin{proof}
We show that a sequence $\left\{\;\prod_{t\in (s+N)\cap S} X_t\;;\;s \in S'\;\right\}$ of \textsc{rv}'s is
 independent for any finite subset $S'=\{s_1,\ldots,s_\nu\}$ of
 $S$. The proof is by induction on $\nu=\#S'$.
When $\nu=1$, it is trivial, so assume $\nu>1$. 
   A set
$(\{s_1,\ldots,s_{\nu-1}\}+N) - (s_\nu+N) =
 \bigcup_{i=1}^{\nu-1} \bigl(s_i+N - (s_\nu+N)\bigr)$ is
 disjoint from $D$ by the premise. So, a random vector $\bigl(
 X_t\;;\; t\in (\{s_1,\ldots, s_{\nu-1}\}+N)\cap S\bigr)$ is
 independent from a random vector $\bigl(X_t\;;\; t\in (s_\nu + N)\cap S
 \bigr)$. Thus a $(\nu-1)$-dimensional random vector $\bigl(\; \prod_{t\in
( s_i + N)\cap S}X_t\;;\; 1\le i\le \nu-1\; \bigr)$ is
independent from an \textsc{rv} $\prod_{t\in (s_\nu+N)\cap S} X_t$. 
 Because a sequence $\bigl\{\; \prod_{t\in
( s_i + N)\cap S}X_t\;;\; 1\le i\le \nu-1\; \bigr\}$ of \textsc{rv}'s
 has length
$\nu-1$ and is independent by the induction hypothesis on
 $\nu$, we are done.
The if-part of the last sentence of the statement is proved by taking $N=\{0\}$, while the only-if part is immediate.
\qed\end{proof}

\begin{definition}
A \emph{finitely randomized model set}
 (\textsc{frms} for short) on a model set $\Lambda(W)$ is a
 random field $\{X_s\}_{s\in\Lambda(W)}$ with a
\emph{finite} d-set $D$.

\end{definition}
The \textsc{frms} can be regarded as a
 Dirac comb with random weights.
When each $X_s$ is an indicator\, (i.e., a $\{0,1\}$-valued \textsc{rv}), 
 we intend that  $X_s>0$  if and only if
 $s\in\Lambda(W)$ indeed appears.

\begin{example}\label{eg:disjunion}
 Suppose $\Gamma = \Lambda(W)$ is a model set such that the star map
 $(-)^\star$ is injective. Suppose a set $C$ is $\{p_1, \ldots, p_n\}\subset L$ such that 
$(W + p_i^\star) \cap  (W + p_j^\star) \cap  L^\star = \emptyset\quad
 (i \ne  j)$. Then 
 $\Lambda=\Lambda(W + C^\star)$ is also a model set, and for every
 $s\in\Lambda$ there are unique $t\in \Gamma$ and unique
 $i\in\{1,\ldots,n\}$ such that 
\begin{align}s=t+p_i. \label{stpi}
\end{align}
Let
$\{\vec{Y}_t\}_{t \in \Gamma}$
be an infinite independent sequence of $n$-dimensional random vectors
 taking values in $ \{0, 1\}^n$. By using \eqref{stpi}, define a random
 field $\{X_s\}_{s \in  \Lambda}$ so that
$X_s$ is the $i$-th component of $\vec{Y}_t$.
Then we can prove that  $\{X_s\}_{s \in \Lambda}$ is indeed a random field by Kolmogorov's consistency theorem~\cite[p.129]{MR1267569}. Below  we explain that
 $D := C - C$ is a d-set of this random field.
Let
 $P, Q$ be finite subsets of $\Lambda$ such that the pair of
a $\#P$-dimensional random vector
$\rvec{X_s}{s \in  P}$ and a $\#Q$-dimensional random vector $\rvec{X_s}{s \in  Q}$ is not independent. Take 
minimal subsets $P', Q'$ of $\Gamma$ such that
$P \subset  P' +  C$ and $Q \subset  Q' +  C$.
Because of the choice of the finite sets $P,Q$, the random vector
 $(X_s)_{s\in P}$ is not independent from the random vector
 $(X_s)_{s\in Q}$. So
the pair of random vectors $\rvec{\vec{Y}_t}{t \in  P'}$
 and $\rvec{\vec{Y}_t}{t \in  Q'}$
is not independent. But the set $\{\vec{Y}_t\}_{t \in \Gamma}$ is independent, so
there is $t \in  P' \cap  Q'$. Then, by the
minimality of  $P', Q'$, there are $p, q \in  C$ such that 
$t + p \in  P$ and  $t + q \in  Q$. Therefore $p - q \in  P - Q$. Thus
$(P - Q) \cap  D \ne  \emptyset$.

%
%
\end{example}

\begin{example}[\textsc{frms} caused by random shift of
 windows]\label{eg:rmsshift}  For a model set
$\Lambda(W)$ with  both of
 $\fG$ and $\Gint$ being Euclidean vector spaces, physicists often
associate to each site $s\in L$ its own window $W_s=W+ y_s$ where
the ``shift'' $y_s$ is an \textsc{rv} 
ranging over a window $R\subset\Gint$ with nonempty interior.
 Then $W+R$ is again a window.
For $s\in\Lambda(W+R)$, define an \textsc{rv}
$X_s$ such that $X_s=1$ whenever the \textsc{rv} $y_s$ takes a value in
 $s^\star-W$, in other words, $s^\star\in W_s=W+y_s$, while $X_s=0$ otherwise.
If the sequence $\{y_s\}_{s\in \Lambda(W)}$ of the \textsc{rv}'s is independent, then the random field
 $\{X_s\}_{s\in\Lambda(W+R)}$ is independent, so  it is an \textsc{frms}
 on a model set $\Lambda(W+R)$. 

     However, if  $y_s(\omega)=y_t(\omega)$ for any $s,t\in L$ and for any
 $\omega$ of the probability space, then the random field is not an \textsc{frms}, because no finite
 d-set  can be taken owing to the existence of a relatively
 dense subset $\Gamma:=\Lambda((W+R)\setminus W)\subset \Lambda(W+R)$
 such that a sequence
 $\{X_s\}_{s\in\Gamma}$ of the \textsc{rv}'s is not independent. Here the
 relative density follows from $Int((W+R)\setminus W)\ne\emptyset$, and
 the proof is in the appendix.

\end{example}

\subsection{Finitely periodic model sets and internal space}\label{subsec:finitely periodic}

Roughly speaking, 
we say a subset $A$ of $H$ is a finitely periodic, if
and only if for each $x\in
(A-A)\setminus\{0\}$ there is a positive integer $\ell$ such that
every arithmetical progression with the
common difference being $x$ has length smaller than $\ell$.
The precise definition is as follows:

\begin{definition}
For any \textsc{lcag} $H$, any set $A\subset H$, any $x\in H$
 and any $s\in A$,  $\ell_A(x\;;s)$ is 0 if $x=0$, or else is the maximum
 $k\in \Zset_{\ge0}\cup\{\infty\}$ such that $\{s -  n x\;;\; 0\le n\le
 k\}\subset A$. So if $x\not\in A-A$, then $\ell_A(x\;;s)$ is $0$ for
 any $s\in A$.
Set
\begin{align*}
\ell_A(x):=\max_{s\in A} \ell_A(x\;;s).
\end{align*} If $\ell_A(x)$ is finite
      for any $x\in A - A$,  we say $A$ is
 \emph{finitely periodic}. We say 
an \textsc{frms} $\{X_s\}_{s\in\Lambda}$ is \emph{finitely periodic}, if
 so is $\Lambda$.
\end{definition}

%
\begin{lemma}
\begin{enumerate}
\item If a model set is finitely periodic,  the star map is
      injective.
\item
A model set is finitely periodic, if the star map is
 injective and the internal space is isomorphic to
 $\Rset^l\times\Zset^m\times F$ for some $l,m\in\Zset_{\ge0}$ and some
 finite Abelian group $F$.
\end{enumerate} 
\end{lemma}

\begin{proof} (1) If the kernel of the star map has nonzero element $x$,
 then for any site $s$ in the model set, we have an infinite arithmetical
 progression $(s+n x)_{n \in\Nset}$.  (2)
Assume the model set $\Lambda(W)$  over a c.-p.\@ scheme $(\fG, \Gint,
 \tilde L)$ is not finitely periodic.  Let $N=\# F$. Then there is $x\in\left(\Lambda(W)-\Lambda(W)\right)\setminus\{0\}$ satisfies that for all
 $\ell\in \Nset$ there is $s_\ell\in\Lambda(W)$ such that 
\begin{align*}
\{s_\ell - n x\;;\;
 0\le n< N \ell \}\subseteq \Lambda(W).
\end{align*}
 Let $\varphi$ be the isomorphism from $\Gint$ to the
 $\Rset^l\times \Zset^m\times F$ and let $\psi_1$ be a homomorphism
 $L\stackrel{\star}{\to}
 L^\star\hookrightarrow\Gint\stackrel{\varphi}{\simeq}\Rset^l\times\Zset^m\times
 F\stackrel{\pi}{\longrightarrow} \Rset^l\times\Zset^m$ where $\pi$ is
 the canonical projection, and $\psi_2$
 be the other homomorphism from $L$ to $F$.  Since
 $F$ is a finite group, $N \psi_2(x)=0$. The injectivity of the star
 map implies $\psi_1(N x)\ne0$.

Then
\begin{align*}
\{\psi_1(s_\ell) - n\psi_1(x)\;;\;
 0\le n< N \ell \}\subseteq \pi\left(\varphi(W)\right).
\end{align*}
We can take the integer $\ell$ greater than $d/d'$ where $d$ is the diameter
 of the compact set $\pi\left(\varphi(Cl(W))\right)\subseteq
 \Rset^l\times \Zset^m$ and $d'$ is the
 norm of $\psi_1\left( N x\right)$.  But this is a contradiction.
\qed\end{proof}

\begin{example} \begin{enumerate}\item
The vertex sets of
 the rhombic Penrose tilings are finitely periodic model sets with
 $\Gint=\Cset\times (\Zset/5\Zset)$ and  
injective star maps. See \cite[Section~3.2]{MR1460032}.

\item If an internal space is the compact Abelian group of $p$-adic integers,
 we can find a model set~\cite{MR1633181} which is not finitely
 periodic.
\end{enumerate}
\end{example}

\subsection{Decomposition of finitely randomized model sets}\label{subsec:decompose}

\begin{definition}\label{lem:sub}
A \emph{cut-and-project subscheme} of a c.-p.\@ scheme $(\fG, \Gint,
\tilde L)$ is a c.-p.\@ scheme $(\fG, \Hint, \tilde M )$ such that  $\tilde M\subset\tilde L$,
$\Hint\subset\Gint$, and the topology of $\Hint$ is the relative topology
      induced from $\Gint$. 
\end{definition}

\begin{lemma}\label{prop:measure} If an \textsc{lcag} $\Hint$ is a
 subgroup of an \textsc{lcag} $\Gint$, then the Haar measure $\theta$ of $\Gint$ restricted
 to $\Hint$ is a Haar measure  $\vartheta$ of $\Hint$ such that
\begin{align*}
\theta\left(\partial_\Gint(B)\right)=0\Rightarrow
\vartheta(\partial_\Hint(B\cap \Hint))=0. 
\end{align*} 
\end{lemma}
\begin{proof}It is immediate that
 $\vartheta$ is a Haar measure of $\Hint$.
Note $\partial_\Hint(B\cap \Hint)$ is $\left(Cl(B\cap \Hint)\cap
\Hint\right)\setminus \left(Int(B\cap\Hint )\cap \Hint\right)$ contained
by $(Cl(B)\cap\Hint)\setminus Int(B) =\partial_\Gint (B)\cap \Hint$. So, if
$\theta\left(\partial_\Gint(B)\right)=0$ then $\theta\left(\partial_\Hint(B\cap
\Hint)\right)=0$, which implies $\vartheta(\partial_\Hint(B\cap
\Hint))=0$ since $\vartheta$ is the restriction of $\Gint$'s Haar
measure $\theta$ to $\Hint$.
\qed\end{proof}

\begin{lemma}\label{lemma:sub}
For any c.-p.\@ scheme $(\fG, \Gint, \tilde L)$ with $L$ being
 finitely generated and for any finite subset $D\subset L$ there are a
 c.-p.\@ subscheme $\S=(\fG, \Hint, \tilde M)$ of $(\fG,\Gint,\tilde L)$
 and a finite complete representation system $R$ of $ L/ M$ with
 $D\subset R$.
\end{lemma}

\begin{proof} By the structure theorem of a
 finitely generated Abelian group, $\tilde{L}$ is isomorphic to $\Zset^v\times
\Zset/\Zset_{n_1}\times\cdots\times\Zset/\Zset_{n_u}$ for some  $u,v\in\Zset_{\ge0}$
 and for some integers $n_1,\ldots,n_u\ge 2$.  Because $D$ is regarded as a finite
 subset of $\Zset^v\times \Zset/\Zset_{n_1}\times\cdots \Zset/\Zset_{n_u}$, there is an integer $k$ greater than the ``first components'' of
 any elements of $D$. Without loss of
 generality, each $n_i$ divides $k$. Then put $M:=k L$.  We can find a finite complete
 representation system $R\supseteq D$. 
 Let $\tilde{M}:=\{(t, t^\star)\;;\; t \in M\}$ and
$\Hint:=Cl(M^\star)\subset\Gint$. Then $\Hint$ is a $\sigma$-compact
 \textsc{lcag} with the relative topology induced from $\Gint$, and $(\fG,
\Hint, \tilde M)$ is a c.-p.\@ scheme with the star map being
a restriction of the star map of $(\fG, \Gint, \tilde L)$. 
\qed\end{proof} 

We consider the following condition on \textsc{frms}'s:
\begin{assumption}\label{cond:1} $\Lambda_\O(W)$ is a finitely periodic model set over
 a c.-p.\@ scheme $\O=(\fG, \Gint, \tilde L)$ with $L$ being
 finitely generated.
\end{assumption}

Recall that $\theta$ is a Haar measure of $\Gint$. The subscheme
$\S$ given in Lemma~\ref{lemma:sub} is used in the following theorem:

\begin{theorem}[Decomposition]\label{thm:partition} Suppose an
 \textsc{frms} $\{X_s\}_{s\in \Lambda_\O(W)}$ satisfies
 Condition~\ref{cond:1}. Then 
 there are a c.-p.\@ subscheme
 $\S=(\fG, \Hint, \tilde M)$ of the c.-p.\@ scheme $\O$ and a finite
 complete representation system $R=\{r_C\in L\;;\;C\in L/ M\}$ of $ L/
 M$ such that following holds:

 for each $g\in\Lambda_\O(W)-\Lambda_\O(W)$, each 
$C\in L/M$,  each $ r\in
 R$ with $r\equiv g \bmod M$,  and  each nonnegative integer $k\le
 \ell:=\ell_{\Lambda_{\O}(W)}(g-r)$, there is a relatively compact set
 $V_{C,k}\subseteq W$ such that for
\begin{align}
S_{C,k}:=r_C + \Lambda_\S(V_{C,k}), \label{def:sck}
\end{align} 
we have
\begin{enumerate}
\item \label{assert:independence}  the sequence $\{X_s \cconj{X_{s-g}}\;;\; s\in
S_{C,k},\ s-g \in \Lambda_\O(W)\;\}$ is independent; and

\item \label{assert:density}  if the Haar measure
 $\theta(\partial W)$ is 0, then  $\dense_{\{D_n\}_n} (S_{C,k})$ exists.
\end{enumerate} 
Furthermore
$\Lambda_\O(W)$ is the disjoint union of all the $S_{C,k}$'s.

\end{theorem}
\begin{proof} Let the subscheme $\S$ and the finite complete
 representation system $R$ be as in Lemma~\ref{lemma:sub}
 applied for the c.-p. scheme $\O$ and a d-set $D$ of the \textsc{frms}
 $\{ X_s\}_{s\in \Lambda_\O(W)}$. Since $0\in D\subseteq R$ by Lemma~\ref{lemma:sub}, if
 $0=g\equiv r \bmod M$
 then $r=0$.  Because $\Lambda_\O(W)$ is finitely periodic, $\ell$ is
finite, and we have relatively compact sets
\begin{align*}
W_k:=\{y \in W\;;\;\ell_W(g^\star-r^\star\;;\;y)=k\}\quad (0\le k\le
 \ell).
\end{align*}
Then $\Lambda_\O(W)=\coprod_{k=0}^{\ell} \{s\in L\;;\; s^\star\in W_k \}$ where
 $\coprod$ is a disjoint union. Since $L=\coprod_{C\in L/M} (r_C + M)$,
 we have $\Lambda_\O(W)=\coprod_{k=0}^{\ell}\coprod_{C\in L/M} \{s \in r_C+M\;;\; s^\star -r_C^\star \in (W_k - r_C^\star)\cap \Hint\}$. Therefore
 $\Lambda_\O(W)=\coprod_{k=0}^{\ell}\coprod_{C\in L/M} \Bigl( r_C +
 \Lambda_\S \bigl(\, (W_k - r_C^\star)\cap \Hint \bigr)\ \Bigr)$. So define
\begin{align}
V_{C,k} := (W_k - r_C^\star)\cap \Hint. \label{def:vcgkl}
\end{align}

Now we prove the assertion~\eqref{assert:independence} of Theorem~\ref{thm:partition} below.
By Lemma~\ref{lem:irms} with $N:=\{-g, 0\}$, it is sufficient to prove the following claim~:  for any distinct $s,
t \in S_{C,k}$, $(\{s, s-g\} - \{t, t-g\})\cap D=\emptyset$.
In other words, for any $g'\in D$ we have (i) $s - t = (s - g) - (t - g) \ne
g'$; (ii) $s - (t - g) \ne g'$; and (iii) $(s - g) - t \ne g'$.  

It is proved as follows:
By Lemma~\ref{lemma:sub}, $M\cap D=\{0\}$, which derives
the
assertion (i) from $s - t \in M \setminus \{0\}$.  If $g=0$, then (ii) and (iii) follow from (i). So let
$g\ne0$.  Assume (ii) is false. Then $g' - g = s - t \in M$, so
 $ g' \equiv g \equiv r \bmod M$ for a unique $r\in R$.
Then $g'=r$, because $g'\in D$ belongs to $R$. Therefore
$t -  s = g - r$.
But  $t^\star -(g^\star -r^\star)=s^\star\in W_k\ni t^\star$ 
 contradicts against the definition of $W_k$.  Therefore (ii)
 holds. The assertion (iii) follows from (ii) since $D=-D$.

We prove the assertion~\eqref{assert:density} of
 Theorem~\ref{thm:partition}. Recall that the \textsc{lcag} $\Hint$ of the
 subscheme $\S$ is a
 subgroup of the \textsc{lcag} $\Gint$. As in Lemma~\ref{prop:measure}, we write
 $\vartheta$ for a Haar measure of $\Hint$.
The premise
 $\theta\left(\partial_\Gint W\right)=0$ implies 
\begin{align}
\vartheta(\partial_\Hint V_{C,k})=0.\label{nullboundary}
\end{align}
To see it first observe that $W_k = W\cap (W + g -r ) \cap \cdots \cap
 (W + k (g - r)) \setminus (W + (k + 1) (g - r)) )$. So the premise implies
 $\theta\left(\partial_{\Gint} (W_k - r_C^\star) \right) = 0$ because of fact
 $\partial_\Gint (A\cap B)\subset\partial_\Gint(A)\cup
 \partial_\Gint(B)$ and fact
 $\partial_\Gint(A)=\partial_\Gint(\Gint\setminus A)$. Thus
 $\eqref{nullboundary}$ follows from Lemma~\ref{prop:measure}.

Since $\{D_n-r_C\}_n$ is a van Hove sequence too,  $\dense_{\{D_n-r_C\}_n} \left(
 \Lambda_\S(V_{C,k})\right)$ converges by 
Proposition~\ref{prop:weyl}. But it is
$\dense_{\{D_n\}_n} \left(S_{C,k}\right)$
by the definition \eqref{def:sck} of $S_{C,k}$.
\qed\end{proof}

\section{Absolutely continuous component of diffraction and covariance}\label{sec:debye}

If a complex-valued
 \textsc{frms} $\omega=\{X_s\}_{s\in\Lambda}$ satisfies
 Condition~\ref{cond:1} and all of the expectation $\E{X_s}$ and the
 covariances between $X_s$ and $X_{s-g}$ are ``continuous'' with
 respect to $s^\star\in\Gint$ for any $g\in \Lambda-\Lambda$, then  we
 quantitatively give the diffraction measure of  $\omega$ as
 follows:
\begin{itemize}
\item the inverse Fourier transform of the
 absolutely continuous component is a Dirac comb whose support is the
 smallest d-set; and
\item the pure-point component 
 is the diffraction measure of a Dirac comb 
      which is the expectation of the \textsc{frms} $\omega$.
\end{itemize} 

Here
\begin{definition} The \emph{expectation} of an \textsc{frms}
 $\omega=\{X_s\}_{s\in\Lambda}$ is, by definition, a Dirac comb
 $\E{\omega}=\{\E{X_s}\}_{s\in\Lambda}$, that is, $\sum_{s\in \Lambda}\E{X_s}\delta_s$.
\end{definition}

We use Kolmogorov's strong law of large numbers. By the
variance of an \textsc{rv} $X$, we mean $\V{X}=\E{\,\left| X -
\E{X}\,\right|^2}= \E{\,|X|^2}- \E{X}\E{\,\cconj{X}\,}$.

\begin{proposition}[\protect{\cite[Corollary~1.4.9]{MR1267569}}]\label{prop:kolmogorov}
Suppose that $\{b_m\;;\;m\in\Nset\}$ is a nondecreasing sequence of positive
 numbers which tends to infinity, and that a set
 $\{X_n\}_{n\in\Nset}$ of square integrable \textsc{rv}'s is
 independent. If $\sum_{m=1}^\infty \V{X_m}b_m^{-2}<\infty$, then
\begin{align*}\lim_{m\to\infty} \frac{1}{b_m} {\sum_{i=1}^m(X_i
-\E{X_i}\,)}=0\qquad(\mbox{almost surely}).
\end{align*} \end{proposition}

\begin{lemma}\label{lem:a}
If a set $A\subset\fG$ is a nonempty and discrete  and  $\{Y_s\}_{s\in A}$ is an independent set of
 \textsc{rv}'s with the variances $\mathrm{V}[Y_s]$ bounded uniformly from
 above, then for any van Hove sequence $\{D_n\}_n$
\begin{align*}
\lim_{n\to\infty}\;\frac{1}{\#(A\cap D_n)}\sum_{s\in A\cap D_n}
 (Y_s - \E{Y_s})\;=0\quad(\mbox{almost surely}).\end{align*}
\end{lemma}
\begin{proof}There is an enumeration  $\{s_i\}_{i\in\Nset}$ of $A$
 without repetition which exhausts $A\cap D_1$ first then $A\cap D_2$,
$A\cap D_3$, and so on, because $A\subseteq \bigcup_n D_n$ and because $\#(A\cap D_n)<\infty$ follows from the
 discreteness of $A$ and the compactness of $D_n$. For $m\in\Nset$ let
$b_m$ be $\#(A\cap D_n)$ where $n$ is the smallest integer such that
 $s_m\in A\cap D_n$. If
 $m=\#(A\cap D_n)$ for such $n$, then $b_m=m$. In general, we have
$b_m\ge m$, because
the van Hove sequence $\{D_n\}_n$ is increasing.
So $\sum_{m=1}^\infty
b_m^{-2}\le\sum_{m=1}^\infty m^{-2}<\infty$. Because $\mathrm{V}[Y_s]$
 is uniformly bounded from above by the premise, we have
$\sum_{m=1}^\infty
{\mathrm{V}[Y_{s_m}]}m^{-2}<\infty$. Since the sequence
$\{Y_{s_i}\}_{i\in\Nset}$ is independent from the premise,
 Proposition~\ref{prop:kolmogorov} implies that the sequence
$\bigl\{\,\sum_{i=1}^m(Y_{s_i} - \E{Y_{s_i}}\,)\;/\;{b_m}\,
\bigr\}_{m\in\Nset}$ converges to 0 almost surely. Hence a
subsequence $\bigl\{\;\sum_{s\in A\cap D_n}(Y_s -
\E{Y_s}\,)\,/\,{\#(A\cap D_n)\;} \bigr\}_{n\in\Nset}$ does so almost surely.
\qed\end{proof}

By the \emph{covariance} between complex-valued \textsc{rv}'s $X_s$ and $X_t$, we mean
\begin{align*}\cov[X_s,\, X_t]= \E{\,(X_s-\E{X_s})\cconj{(X_t-\E{X_t})}\,}=\E{X_s\cconj{X_t}}- \E{X_s}\E{\cconj{X_t}}.\end{align*}

\begin{assumption}\label{cond:envelope} An \textsc{frms}  $\{X_s\}_{s\in \Lambda(W)}$
 has functions $e, c_g\in \cbp$ such that
\begin{align*}
\E{X_s}&=e(s^\star),\quad      &\left(s\in\Lambda(W)\;\right);\\
\cov[X_s,\,X_{s-g}]&=c_g(s^\star),&\left(g\in \Lambda(W)-\Lambda(W),\
 s\in \Lambda(W)\cap (\Lambda(W) + g)\right).
\end{align*}
\end{assumption}

\begin{example}In Example~\ref{eg:disjunion}, suppose that there are $m\in C(\Gint, \Rset^n)$ and $S\in C(\Gint, \Rset^{n\times
 n})$ such that 
$m(t^\star)$ is the mean $\E{\vec{Y}_t}$ and $S(t^\star)$ is the
 covariance matrix $\E{(\vec{Y}_t - \E{\vec{Y}_t})(\vec{Y}_t -
 \E{\vec{Y}_t})^\top}$ for all $t\in \Gamma$.
 Then functions $e,c_g$ $(g\in \Lambda-\Lambda)$ indeed belong to $\cbp$.
For the \textsc{frms} of Example~\ref{eg:rmsshift},
assume further that each shift $y_s$ is subject to a continuous probabilistic density
function $h$ with $\supp{h}= R$. Then,
for $s\in L$, $(1_W\ast h)(s^\star) = \int_\Gint 1_W(s^\star- y)h(y)d y = \int _{W\ni s^\star- y}h(y)d y =P(s^\star\in W +
y_s)$, the probability for $s\in \Lambda(W+R)$ to indeed appear. So $e=1_W\ast h$. 
\end{example} 

\begin{theorem}\label{thm:main}
Let $\omega=\{X_s\}_{s\in \Lambda(W)}$ be an \textsc{frms} such that
Condition~\ref{cond:envelope} holds, $\{\,\E{\,|X_s|^4}\,\}_{s\in\Lambda(W)}$ is
 bounded, and $W$ is compact but the Haar measure 
 $\theta(\partial W)$ of $\Gint$ is 0. 
Then the  diffraction measure $\widehat{\gamma_\omega}$ of
 $\omega$ is almost surely
$\dual{\gamma_{\E{\omega}}}\ +\ A$,
where 
\begin{enumerate}
\item $\dual{\gamma_{\E{\omega}}}$ is a pure-point diffraction measure.
\item $A$ is an absolutely continuous, real-valued measure on $\dual
 \fG$. In fact, there is some d-set  $D$ of the \textsc{frms}
      $\omega$  such that the Radon-Nikod\'ym derivative of $A$ with respect to the Haar
      measure $d\chi$ of $\dual\fG$ is $\sum_{g\in D} A_g
      \chi(-g)$ where 
\begin{align*}
A_g
= \int_{W \cap (W+g^\star)} c_g(y)d y =
 \lim_{n\to\infty}\sum_{s\in\Lambda(W)\cap (\Lambda(W)+g)\cap
 D_n}\frac{\cov[X_s ,\, X_{s-g}\,]}{|D_n|}.
\end{align*}

\end{enumerate}
\end{theorem}

\begin{proof} We use the notation of Theorem~\ref{thm:partition}.
Let $g\in \Lambda_\O(W)-\Lambda_\O(W)$. Because $\Lambda_\O(W)$ is the
 disjoint union of $S_{C,k}$'s over $(C,k)\in (L/M)\times\{0,\ldots,
 \ell\}$,  we have
$s,s-g\in \Lambda_\O(W)\cap D_n$ if and only if there is a
 unique $(C,k)\in (L/M)\times\{0,\ldots, \ell\}$ such
 that $s\in S_{C,k}\cap (\Lambda_\O(W)+g)\cap D_n\cap (D_n+g)$. Furthermore
 $r_C -g \in R - M$ and 
\begin{align} \label{eq:intersection} 
&S_{C,k}\cap (\Lambda_\O(W)+g) = r_C +     \Lambda_\S(W_{C,k}), \end{align}
where
\begin{align}
&W_{C,k}:= V_{C,k} \cap
\bigl( (W-r_C^\star+g^\star)\cap \Hint\bigr). \label{claim:1:2}
\end{align}

Thus, 
\begin{align} \label{claim:1:3}
\lim_{n\to\infty} \sum_{s\in \Lambda_\O(W)\cap D_n }
\frac{(\cdots)}{|D_n|}
=
\sum_{\scriptsize\begin{array}{l}C\in (L/M), \\ 0 \le k \le \ell\end{array}}\
 \lim_{n\to\infty} \sum_{\scriptsize \begin{array}{lll}s&\in& (r_C + \Lambda_\S(W_{C,k}))\\ & \cap& D_n
\cap (D_n+g)\end{array}} \frac{(\cdots)}{|D_n|}. 
\end{align}

Suppose $C\in L/M$, $0 \le k\le \ell$ and
 $S_{C,k}\ne\emptyset$. 

\begin{myclaim}\label{claim:b}For any $f\in \cbp$, as $n\to\infty$, both of the summation of
$f(s^\star)|D_n|^{-1}$ over $s\in (r_C + \Lambda_\S(W_{C,k})) \cap D_n\cap (D_n+g)$
and the summation of $f(s^\star)|D_n|^{-1}$ over $s\in (r_C + \Lambda_\S(W_{C,k}))\cap D_n$ converge to the same value.
\end{myclaim}

\begin{proof}
The absolute value $\Delta$ of the difference between the two summations is not
greater than $ \sum_{s\in (r_C + \Lambda_\S(W_{C,k}))
\cap (D_n\setminus (D_n+g))} |f(s^\star)|\,/\, |D_n|$. Here, by \eqref{eq:intersection}, a set  $r_C+
\Lambda_\S(W_{C,k})$ is a subset of $\Lambda_\O(W)+g$, and is uniformly
 discrete.
Thus by Lemma~\ref{lem:ud}, there is a compact neighborhood $U$ of $0$
such that for all $n$, $\#\Bigl(\bigl(r_C + \Lambda_\S(W_{C,k})\bigr) \cap
\bigl(D_n\setminus (D_n+g)\bigr)\Bigr)\le \left|D_n\setminus
(D_n+g)\right|\,/\,|U|$. By the premise $f\in \cbp$, there is $b\ge 0$ such
 that $|f(y)|\le b$ for all
 $y\in \Gint$. Thus $\Delta\le{\left|D_n\setminus(D_n-g)\right|}\cdot {b}{|U|}^{-1}\,/\,{|D_n|}
={b}{|U|}^{-1} {\left|(D_n+g)\setminus D_n\right|}\,/\,{|D_n|} \le
{b}|U|^{-1}\, {\left|\partial^{\{0, -g\}}(D_n)\right|}\,/\,{|D_n|}\to 0$.
Moreover, by Proposition~\ref{prop:weyl}, the summation of
 $f(s^\star)\;/\;|D_n|$ over
${s\in (r_C + \Lambda_\S(W_{C,k}))
\cap D_n} $ converges as $n$ goes to infinity. To see it, the
 definition~\eqref{claim:1:2} of $W_{C,k}$
implies $\partial_\Hint(W_{C,k})$ is contained by $\partial_\Hint(V_{C,k})\cup
\partial_\Hint\left( (W-r_C^\star +g^\star) \cap \Hint\right)$ which has null
measure with respect to the Haar measure $\vartheta$ of $\Hint$, where
$\vartheta\left(\partial_\Hint(V_{C,k})\right)=0$ by
\eqref{nullboundary}, while the premise $\theta(\partial W)=0$ of Theorem~\ref{thm:main} implies that $
(W-r_C^\star+g^\star) \cap \Hint$ has $\vartheta$-null boundary by
Lemma~\ref{prop:measure}.\qed\end{proof}

\begin{myclaim}\label{claim:3}
It holds almost surely that as $n$ goes to $\infty$, 
\begin{align}
  \sum_{\scriptsize\begin{array}{lll}s & \in &(r_C + \Lambda_\S(W_{C,k})) \\
                           &&\cap\; D_n\cap\; (D_n+g)\end{array}}
\frac{X_s \cconj{X_{s-g}}}{|D_n|}
\ \ -
  \sum_{\scriptsize\begin{array}{lll}s & \in &(r_C + \Lambda_\S(W_{C,k})) \\
                           &&\cap\; D_n\cap\;
		   (D_n+g)\end{array}}\frac{\mathrm{E}\left[X_s
 \cconj{X_{s-g}} \,\right]}{|D_n|} \to 0.
\label{eq:8}
\end{align}
\end{myclaim} 
\begin{proof}
Note that the left-hand side of \eqref{eq:8}
 is the product of the following two: 
\begin{align*}
u_n &:= \sum_{s\in (r_C + \Lambda_\S(W_{C,k})) \cap
 D_n\cap (D_n+g)}\frac{X_s \cconj{X_{s-g}}- \mathrm{E}\left[X_s
 \cconj{X_{s-g}}\,\right]}{ \# \left\{ (r_C + \Lambda_\S(W_{C,k})) \cap
 D_n\cap (D_n+g) \right\}},\\
v_n&:={\#\left\{(r_C + \Lambda_\S(W_{C,k})) \cap
 D_n\cap (D_n+g)\right\}} / |D_n|.
\end{align*}
Here $u_n$ tends to
 $0$ almost surely, because by the premise 
 the variances 
$\V{X_{s} \cconj{X_{s-g}}}\le \E{\,\left|X_{s}\cconj{X_{s-g}}\right|^2}\le 
\E{\,|X_{s}|^4}/2+\E{\,|X_{s-g}|^4}/2$ are uniformly bounded, 
from which 
 Theorem~\ref{thm:partition}~\eqref{assert:independence}, the equality~\eqref{claim:1:3} and
 Lemma~\ref{lem:a} imply that $\lim_{n\to\infty} u_n=0$ almost surely.
On the other hand, $v_n$ tends to $\dense_{\{D_n\}_n} \
 \Lambda_\S(W_{C,k})<\infty$ by Claim~\ref{claim:b}.  \qed
\end{proof} 

The second sum in the left-hand side of \eqref{eq:8} has the following limit
in the limit of $n$.

\begin{myclaim}\label{claim:4} As $n$ goes to infinity, $\sum_{\scriptsize\begin{array}{lll}s & \in &(r_C + \Lambda_\S(W_{C,k})) \\
                           &&\cap\; D_n\cap
		   (D_n+g)\end{array}}\fra{\mathrm{E}\left[X_s
 \cconj{X_{s-g}} \,\right]}{|D_n|}$
converges to
\begin{align}
 \lim_{n\to\infty} \sum_{\scriptsize\begin{array}{lll}s & \in &(r_C + \Lambda_\S(W_{C,k})) \\
                           &&\cap\; D_n\;\cap
		   (D_n+g)\end{array}} \frac{e(s^\star)\cconj{e(s^\star -g)}}{|D_n|}  
+ \lim_{n\to\infty}
 \sum_{\scriptsize \begin{array}{lll}s &\in &(r_C + \Lambda_\S(W_{C,k})) \\ &&\;\cap\;  D_n\end{array} } \frac{c_g(s^\star)}{|D_n|} . \label{eq:9}
\end{align}
\end{myclaim}
\begin{proof} By Condition~\ref{cond:envelope},
 Claim~\ref{claim:b} and Proposition~\ref{prop:weyl}, the
 two limits in \eqref{eq:9} are convergent, and \eqref{eq:9} is
\begin{align*}
\lim_{n\to\infty} \frac{1}{|D_n|} \sum_{s\in (r_C + \Lambda_\S(W_{C,k})) \cap D_n\cap
 (D_n+g)} \left(e(s^\star)\cconj{e(s^\star -g^\star)} +
 c_g(s^\star)\right).\end{align*} But the summand is $\E{X_s
 \cconj{X_{s-g}}\,}$ by Condition~\ref{cond:envelope}.
\qed\end{proof}

By taking the summation \eqref{eq:8} and \eqref{eq:9} respectively over any $(C,k)\in (L/M)\times\{0,\ldots, \ell\}$ such that $S_{C,k}\ne\emptyset$,  by
\eqref{claim:1:3}, we have
 almost surely $ \eta_\omega(g)= \eta_{\E{\omega}}(g) + A_g$, and thus
 $\widehat{\gamma_\omega}= \widehat{\gamma_{\E{\omega}}} + \sum_{g\in S
 - S} A_g\chi(-g)$.

\medskip
We continue the proof of Theorem~\ref{thm:main}.

The assertion~(1) of Theorem~\ref{thm:main} holds because the
 pure-pointness of $\E{\omega}$  follows from Proposition~\ref{prop:modelsetpp} and
 Condition~\ref{cond:envelope}.

The assertion~(2) is proved as follows: The Radon-Nikod\'ym derivative
 $A(\chi)=\sum_{g\in S - S}A_g\chi(-g)$ is actually  $A(\chi)=\sum_{g\in
 D}A_g\chi(-g)$ for any d-set $D$ of the \textsc{frms} $\omega$. Indeed  $\cov[X_s,\,
X_{s - g}]$ is equal to $0$ for any $g\in
(S - S)\setminus D$  and for any $s\in \Lambda_\O(W)\cap (\Lambda_\O(W)+g)$, because a pair of $X_s$
 and $X_{s-g}$ is independent for this $g$ by Definition~\ref{def:depset}.
This completes the proof of Theorem~\ref{thm:main}.\qed\end{proof} 

\begin{corollary}Under the same assumption as Theorem~\ref{thm:main},
 the \emph{smallest} d-set of the \textsc{frms} is the
 support of the Dirac comb which is obtained by the inverse Fourier transform of
 the absolutely continuous component $A$ of the diffraction measure
$\widehat{\gamma_\omega}$.\end{corollary}

\begin{proof}The inverse Fourier transform of $A$ is $\int_{\dual\fG}
 \sum_{g\in D} A_g \chi(-g) \chi(x) d\chi$ which is $\sum_{g\in D}\int_{\dual\fG} A_g \chi(-g+x) d\chi =\sum_{g\in
 D}A_g\delta_g(x)$. 
\qed\end{proof}

\begin{example}
If the \textsc{frms} in Theorem~\ref{thm:main} is independent,
 it has a d-set $D=\{0\}$ by
Lemma~\ref{lem:irms} and  the absolutely continuous component of the
 diffraction measure
 $\widehat{\gamma_\omega}$ is
$A= \lim_n \sum_{s\in \Lambda(W)\cap D_n}{\V{X_s}}/{|D_n|}$.
\end{example}

If we add a mild condition ``$Cl(Int(W))=W$'' to the theorem, we can
quantitatively provide the pure-point component $\widehat{\gamma_{\E{\omega}}}$ by
using a following theorem (Theorem~\ref{thm:2}) (and can dispense with Proposition~\ref{prop:modelsetpp}.)

\section{Diffraction of weighted Dirac comb and torus parametrization}\label{sec:intensity}

Let $\omega$ be a weighted Dirac comb on $\Lambda(W)$ where $\Lambda(W)$ is a 
c.-p.\@ set over a c.-p.\@ scheme $\S=(\fG,
\Gint, \tilde L)$. To describe the support of $\widehat{\gamma_\omega}$, we use the
dual c.-p.\@ scheme  $(\dual\fG, \dual\Gint, \tilde\cL)$ for
the c.-p.\@ scheme $\S$. Here the lattice $\tilde \cL$ is the \emph{annihilator} of $\tilde
L$ in $\widehat{\fG\times \Gint}\simeq\widehat{\fG} \times \widehat{\Gint}$. That
is $\tilde \cL$ is the
\textsc{lcag} of $(\chi, \eta) \in \dual\fG \times \dual\Gint$ such that
$ \chi(s)\eta(s^\star) = 1$ for all $(s, s^\star) \in \tilde L$. Then
$(\dual\fG, \widehat{\Gint}, \tilde\cL)$ is indeed a c.-p.\@ scheme. See
\cite[Section~5]{MR1460032} for the proof. Let $\cL$ ($\cL^\star$ resp.)
stand for the canonical projection of $\tilde\cL$ to $\dual \fG$
($\widehat{\Gint}$ resp.), and let the star map be
$(-)^\star:\cL\to\cL^\star$. Then
\begin{align}
\chi(s)=\cconj{\chi^\star(s^\star)}\qquad(\chi\in\cL,\ s\in L). \label{juji}
\end{align}
Let $\fT$ be a compact Abelian group $(\fG\times\Gint)/\tilde L$.   Then
the lattice $\tilde{\mathcal{L}}$ of the c.-p.\@
scheme satisfies
\begin{align}
\tilde{\cL}\simeq \widehat{\fT}, \label{isomorphic}
\end{align}
and the composite $\varpi:\widehat{\fT}\to \cL$ of the isomorphism and the projection restricted to
$\widetilde \cL$ is indeed a bijection:
\begin{align}
  \varpi(\xi) := \xi\left( (\bullet, 0) +
 \tilde{L}\right) \in \cL. \label{bijectionvarpi}
\end{align}
Recall that $|\tilde L|=\int_\fT d x d y$.

We say an \textsc{flc} set $P\subset \fG$ is \emph{repetitive} if for every
  compact set $K\subset\fG$ there exists compact set $K'\subset \fG$
  such that for all $t_1,t_2\in \fG$ there exists $s\in K'$ such that
  $(t_1+P)\cap K=(s+t_2+P)\cap K$.  According to \cite{MR1798991}, we
  say a window $W\subset\Gint$ has \emph{no nontrivial translation
  invariance} if $\{ c\in \Gint\;;\; c+W=W \}=\{0\}$.

\begin{theorem}\label{thm:2}Assume that $\Lambda(W)$ is a repetitive model
 set over a c.-p.\@ scheme $(\fG, \Gint, \tilde L)$ where
 $\Lambda(W)-\Lambda(W)$ generates $L$, $W=Cl(Int(W))$, $W$ has no
 non-trivial translation invariance, and the Haar measure
 $\theta(\partial W)$ is $0$. Assume
 furthermore that
\begin{align}
b\in\cbp,\  \supp b\subset W,\ \omega=\sum_{s\in\Lambda(W)}
 b(s^\star)\delta_s. \label{thm:2:hosi}\end{align}
Then the diffraction measure of $\omega$ is 
$\displaystyle \dual{\gamma_\omega} =
 \sum_{\chi\in\L}\frac{| \hat{b}(\chi^\star) |^2
 }{|\tilde{L}|^2} \delta_\chi$.
\end{theorem}

The theorem with the physical space $\fG$ being $\Rset^n$ was proved by
Hof~\cite{MR1328260}, and that for Dirac combs with 
constant weights was by Schlottmann~\cite{MR1798991}. A general theorem for weighted Dirac comb with the weights arising from an
``admissible'' Radon-Nikod\'ym derivative of the $\tilde L$-invariant
measure was studied in Lenz-Richard~\cite[Theorem~3.3]{MR2289878}. Our
theorem is another form of a weaker version of Lenz-Richard's theorem.

In order to prove the theorem, we employ a uniquely ergodic dynamical system $X_{\Lambda(W)}$
made from $\Lambda(W)$, and connection of the autocorrelation measure $\gamma_\omega$ to a complex Hilbert space over
$X_{\Lambda(W)}$. Then we prove lemmas about the so-called \emph{torus
parametrization} of $\Lambda(W)$, introduced in
\cite{0305-4470-30-9-016} and generalized by \cite{MR1798991}.
%

\begin{proposition}[\protect{\cite{MR1798991}}]
\label{prop:ds}
\begin{enumerate}
\item  For every \textsc{flc} set $\Lambda\subset\fG$,
the closure 
$\XP0$ of the $\fG$-orbit $\{ \Lambda+g\;;\;
g\in\fG\}$ of $\Lambda$ with some uniform topology  is a complete, compact Hausdorff space.

\item \label{assert:uniquelyergodic}
Suppose
  a model set $\Lambda:=\Lambda(W)$ satisfies the same assumptions as
       Theorem~\ref{thm:2}. Then 
$\XP0$ will be a minimal and uniquely ergodic dynamical system, with
the  $\fG$-action  $(x, P)\in \fG\times \XP0\mapsto P+x\in \XP0$. 
\end{enumerate}
Hereafter, 
the uniquely ergodic probability
measure of $\XP0$ will be denoted by $\nu$, and
 the complex Hilbert
 space over $\XP0$ with the inner product 
\begin{align*}\inner{\Phi_1}{\Phi_2}_\nu:=\int_{\XP0}
\Phi_1(P) \cconj{\Phi_2(P)}d\nu(P)\qquad(\Phi_i\in L^2(\XP0,\nu))
\end{align*}
will be denoted by $L^2(\XP0, \nu)$.
\end{proposition}

\begin{proposition}[Torus parametrization]\label{prop:beta}
Assume the same assumptions as Theorem~\ref{thm:2}, and 
let $\fT$ be a compact Abelian group $(\fG\times
 \Gint)/\tilde L$ with the Haar probability measure $\tau$,
and let $\fG$ act on $\fT$ by $(x,t)\in
 \fG\times\fT\mapsto t+(x,0)\in \fT$. Then
there are a continuous surjection
\begin{align*}
\beta:\XP0\to\fT
\end{align*} and a full measure subset $\XPa$ of $\XP0$
such that (1) $\beta$ preserves the $\fG$-action, (2)
$\beta(\Lambda)=\tilde L$, and (3)
$\beta':=\beta|_{\XPa}$ is injective with the range $\beta(\XPa)$
 disjoint from $\beta(\XP0\setminus\XPa)$.
\end{proposition}

From the proposition, we can derive a following:
\begin{lemma}\label{proposition:iota}
 Let $L^2(\fT, \tau)$ be a complex Hilbert space with the inner product 
\begin{align*}
\inner{\alpha_1}{\alpha_2}_\tau:=\int_\fT \alpha_1(t)
 \cconj{\alpha_2(t)}d\tau(t)\qquad\left(\alpha_1, \alpha_2\in L^2(\fT, \tau)\right).
\end{align*} 
 Then we have a bijective isometry
\begin{align*}
\iota\ :\ L^2(\XP0,\, \nu)\to L^2(\fT, \tau)\ \ ;\ \ \Theta\mapsto 
\Theta\circ (\beta')^{-1} \enspace.
\end{align*} 

\end{lemma}

By using \cite[Proposition~1.4]{MR0621876}, we can prove the following:
\begin{fact}\label{fact:1.2} 
If $\lambda$ and $\mu$ are translationally bounded, nonnegative measures on $\fG$ and
 $\{D_n\}_n$ is a van Hove sequence on $\fG$, then    in the vague topology
\begin{align*}
\lim_{n\to\infty}\frac{\widetilde{\left(\lambda|_{D_n}\right)} * \mu|_{D_n} - \widetilde{\lambda} * (\mu|_{D_n})}{|D_n|}=\lim_{n\to\infty}\frac{\widetilde{\left(\lambda|_{D_n}\right)} *\mu - \widetilde{\lambda} * (\mu|_{D_n})}{|D_n|}=0.
\end{align*}
\end{fact} 

In the following Lemma,  recall that $\omega$ is a weighted Dirac comb mentioned in
Theorem~\ref{thm:2} and the weight depends on a function $b\in C_c(\Gint)$.
\begin{lemma} \label{lem:abu}Suppose the same assumptions as
Theorem~\ref{thm:2}. Let $\Lambda$ be $\Lambda(W)$. Then, for all
 $\varphi_1,\varphi_2\in C_c(\fG)$, there are \emph{unique} $\Phi_1,\Phi_2\in
 L^2(\XP0, \nu)$ such that for any $x\in\fG$,
\begin{align}
\Phi_i(\Lambda - x) =(\varphi_i * \omega)(x),\label{conv}\\
(\varphi_1 * \widetilde{\varphi_2} * \gamma_\omega)(0) =
 \inner{\Phi_1}{\Phi_2}_\nu.\label{eq10}
\end{align}
\end{lemma}

\begin{proof}Let $b\in \cbp$ be as in Theorem~\ref{thm:2} and
 $\beta:\XP0\to\fT$ be as in Proposition~\ref{prop:beta}. Define
 $g_i\in \cb{\fG\times\Gint}$ by $g_i(x,y):= \varphi_i(-x)b(y)$. Put
 $K_i:= \supp g_i$. Let $P\in \XP0$. Because $\beta(P)\in X_\Lambda$ is a
 discrete set, $\beta(P)\cap K_i$ is finite. Thus
 $\Phi_i(P):=\sum_{(x,y)\in \beta(P)} g_i(x,y)$ is indeed a finite sum
 and well-defined.

  To establish $\Phi_i\in L^2(\XP0, \nu)$,
it is sufficient
 to verify that  a function $f_i:\fT\to\Cset\;;\;t\mapsto \sum_{(x,y)\in t}
g_i(x,y)$ is continuous, because $\Phi_i$ is the composition of the continuous function  $\beta$
 and $f_i$.

Because $g_i\in \cb{\fG\times\Gint}$, for all
 $\varepsilon>0$ there exists a compact neighborhood $U\subset
 \fG\times\Gint$ of $0$ such that for  all $z,z'\in \fG\times \Gint$ with
 $z'\in z+ U$, we have $|g_i(z')-g_i(z)|<\varepsilon$. Then $\{z + \tilde
 L\;;\;z\in U\}$ is a compact neighborhood of 0 in $\fT$. Let $t',t\in
 \fT = (\fG\times \Gint)/\tilde{L}$ such that $t'-t$ belongs to the compact neighborhood $\{z + \tilde
 L\;;\;z\in U\}$. Then $t'=t+(u,v)$ for some $(u,v)\in U$. So
 $|f_i(t')-f_i(t)|\le \sum_{z\in t}\left| g_i\left(z+(u,v)\right) -
 g_i(z)\right|$. We show it is less than $\#\left(((K_i-U)\cup K_i)\cap
 t\right)\times \varepsilon$.  If $z\in t$ contributes to the summand,
 then $z+(u,v)\in K_i$ or $z\in K_i$, so $z\in \left((K_i-U)\cup
 K_i\right)\cap t$. Here $\left((K_i-U)\cup K_i\right)\cap t$ is finite
 since $t$ is a translation of the lattice $\tilde L$ and $(K_i-U)\cup
 K_i$ is compact. Thus $f_i$ is continuous.

\medskip
To prove \eqref{conv}, observe 
\begin{align}
\Phi_i(\Lambda - x)=\sum_{s\in
 L}\varphi_i(-s+x) b(s^\star)\label{eq:sum}
\end{align} follows from $\beta(\Lambda-x)=\tilde
  L-(x,0)$. By the condition $\supp b\subset W$ of \eqref{thm:2:hosi},
 the weight $b(s^\star)$ vanishes for any $s\not\in \Lambda(W)$. So the range
 $L$ of $s$ in the summation \eqref{eq:sum} can be replaced with $\Lambda$. Thus
 \eqref{conv} holds.

\medskip
The left-hand side of \eqref{eq10} is
$
\lim_n { \left(\varphi_1 * \widetilde{\varphi_2} *
\omega|_{D_n} * \widetilde{\omega|_{D_n}}\right)(0)}\cdot |D_n|^{-1}
$ by  definition of $\gamma_\omega$.
 But by Fact~\ref{fact:1.2}, it is 
$\lim_n \Bigl( \left(\widetilde{\varphi_2 * \omega}\right) *
(\varphi_1 * \omega)|_{D_n}\Bigr)(0)\cdot |D_n|^{-1}$, which is equal to
$\lim_n  \int_{D_n}\cconj{\Phi_2(\Lambda - x)}\Phi_1(\Lambda - x)
d x\cdot |D_n|^{-1}$ by \eqref{conv}.
 By the pointwise ergodic theorem~\cite{MR2186253}, it is
$ \int_{\XP0} \cconj{\Phi_2(P)}\Phi_1(P)
 d\nu(P)=\inner{\Phi_1}{\Phi_2}_\nu\enspace
$.
\qed\end{proof}

Here is a technical lemma concerning 
van Hove sequences and uniformly discrete sets.

\begin{lemma}\label{aki2} For any uniformly discrete subset $\Lambda$ of $\fG$,
 any bounded complex sequence $\{w_s\}_{s\in\Lambda}$, any $\chi\in\dual
 \fG$,  any
 $\varphi\in C_c(\fG)$, and any van Hove sequence $\{D_n\}_{n}$ of $\fG$,
\begin{align*}
\frac{1}{|D_n|} \left|\sum_{s\in \Lambda} w_s 
 \int_{D_n} \chi(x)  \varphi(x - s) d  x-  \sum_{s\in \Lambda\cap D_n} w_s
 \int_{\fG} \chi(x)  \varphi(x - s) d  x\right|
\end{align*}
tends to $0$, as $n$ goes to $\infty$.
\end{lemma}

\begin{proof}The numerator is bounded from above by
\begin{align}
\sum_{s\in \Lambda\cap
  D_n} |w_s| \int_{\fG\setminus D_n} |\varphi(x - s)| d x \ \ 
 + \sum_{s\in \Lambda\setminus D_n} |w_s| \int_{D_n} |\varphi(x - s)| d x. \label{contr}
\end{align}
The summands
 in the former summation and the latter summation are  bounded,
because so are the sequence $\{w_s\}_{s\in\Lambda}$ and $\varphi\in C_c(\fG)$. 
So it is sufficient to show that the set of $s$ that ``contributes'' to
 $\eqref{contr}$ has density 0 with respect to $\{D_n\}_{n\in\Nset}$.

Let $s\in\Lambda$ ``contribute'' to \eqref{contr}. 
If $s\in \Lambda$ ``contributes'' to the former summation, then $s\in D_n$
 for some $x\in \fG\setminus D_n$ such that $x - s\in \supp \varphi$. So  
$ s\in
 \left[Cl \left(\fG \setminus D_n\right) - \supp \varphi
 \right] \cap D_n
\subset  \partial^{\supp \varphi} (D_n)$,
by the definition \eqref{thick boundary} of $\partial^{\supp \varphi} (D_n)$. Similarly, if $s\in \Lambda$ ``contributes'' to 
 the latter summation then $s\not\in D_n$ for some $x\in D_n$ such
 that $x - s\in\supp\varphi$, so 
$s \in  ( D_n - \supp \varphi) \setminus Int(D_n)\subset
 \partial^{- \supp \varphi} (D_n)$. Then $s\in
 \partial^{\supp\varphi}(D_n)\cup
 \partial^{-\supp\varphi}(D_n)\subset\partial^K(D_n)$ for some compact set $K\subset\fG$ such that $K=-K\supseteq
 \supp \varphi$.

Thus we have only to verify 
$\lim_{n\to\infty} {\#\{  s\in \Lambda\;;\; s\in \partial^K(D_n)\}}/{|D_n|}  =
 0$. 
This is proved as follows: By Lemma~\ref{lem:ud}, there is a compact set
 $U=-U$ such that  $(s+U)\cap (t+U)=\emptyset$
 for any distinct points $s,t\in\Lambda$. Thus
 $U+\{ s\in \Lambda\;;\;
 s\in \partial^K(D_n)\}\subset \partial^{K+U}(D_n)$, and the
 Haar measure of the left-hand side of the inclusion is 
$|U|\cdot \#\{s\in \Lambda\;;\; s\in \partial^K(D_n) \}$. Thus
$ |U| \cdot \lim_{n\to\infty} \#\{  s\in \Lambda\;;\; s\in \partial^K(D_n)\}/
 |D_n|\le\lim_{n\to\infty} \left|\partial^{K+U}(D_n)\right|/|D_n|=0
 $ by the definition of van Hove sequence.
\qed\end{proof}

\begin{lemma}\label{thm:a} Suppose the same assumptions as in
 Theorem~\ref{thm:2}. Then 
\begin{enumerate}
\item \label{thm:a:1}
For a complex Hilbert space
 $L^2(\dual \fT)$ with the inner product 
\begin{align*}
\inner{\kappa_1}{\kappa_2}_{\dual\fT}:=\sum_{\xi\in\dual \fT}
\kappa_1(\xi) \cconj{\kappa_2(\xi)}\qquad(\kappa_1,\kappa_2\in \dual\fT), 
\end{align*} 
we have a bijective isometry
$\displaystyle  L^2\left( X_\Lambda, \nu\right)\stackrel{\iota}{\to} L^2(\fT, \tau)\stackrel{\dual{(\bullet)}}{\to}
 L^2(\dual\fT)$.

\item  \label{thm:a:assert:2}
For $\varphi_1,\Phi_1$ of Lemma~\ref{lem:abu}, and the projection
      $\varpi$ given in \eqref{bijectionvarpi}, we have
\begin{align*}
\dual{(\iota\Phi_1)}(\xi)  = \frac {
 \widehat{\varphi_1}\left(-\varpi(\xi)\right)}{|\tilde{L}|}
{\hat{b}\left(\varpi(\xi)^\star\right)}\qquad(\xi\in\dual\fT).
\end{align*}
\end{enumerate}
\end{lemma}

\begin{proof} \eqref{thm:a:1}  By \eqref{isomorphic}, $\dual\fT$ is discrete, so the inner product of the Hilbert space $L^2(\dual\fT)$
 is in fact a summation. 
Recall that Fourier transform is a bijective isometry. So the conclusion
 follows from Lemma~\ref{proposition:iota}. 

\eqref{thm:a:assert:2} 
%
Put $\chi:=\varpi(\xi)$. The left-hand side $\widehat{( \iota
 \Phi_1)}(\xi)$ of the equation is  a Fourier transform $\int_\fT \cconj{\xi(t)}
 (\iota\Phi_1)(t) d \tau(t)$, which is $\int_{\XP0} \cconj{\xi \circ \beta(P)}\ 
 \Phi_1(P) d \nu(P)$ by Lemma~\ref{proposition:iota}. 
By Proposition~\ref{prop:ds}~\eqref{assert:uniquelyergodic} and  the pointwise ergodic theorem~\cite{MR2186253},  the integral is
$\lim_{n\to\infty}\int_{D_n}
 \cconj{\xi\left(\beta(\Lambda - x)\right)} \Phi_1(\Lambda - x) d
 x\cdot |D_n|^{-1}$.
Since $\beta(\Lambda-x)=\tilde L - (x,0)$ by
 Proposition~\ref{prop:beta}, we have 
$\xi\left(\beta(\Lambda - x)\right) = \xi(\tilde L -(x,0)) =
 \varpi(\xi)(-x)=  \cconj{\chi}(x)$. So
 $\widehat{\left(\iota\Phi_1\right)}(\xi) = \lim_{n\to\infty}
 \int_{D_n} \chi(x) \Phi_1(\Lambda-x) d x \cdot |D_n|^{-1}$. By  \eqref{conv} of Lemma~\ref{lem:abu} and  the premise
 $\omega=\sum_{s\in\Lambda(W)} b(s^\star)\delta_s$, we have
\begin{align*}
\dual{(\iota\Phi_1)} (\xi) =
 \lim_{n\to\infty}\frac{1}{|D_n|}  \int_{D_n} \chi(x)
\sum_{s\in\Lambda(W)}\varphi_1(x-s) b(s^\star ) d  x. 
\end{align*} 
 We can apply Lemma~\ref{aki2} to above,
since  $\Lambda(W)$ is
 uniformly discrete, the sequence
 $\left\{b(s^\star)\right\}_{s\in \Lambda(W)}$ of weights is bounded, and $\varphi_1\in
 C_c(\fG)$.
Thus
\begin{align*}
\dual{(\iota\Phi_1)} (\xi) =
 \lim_{n\to\infty}\frac{1}{|D_n|} 
\sum_{s\in\Lambda(W)\cap D_n} b(s^\star) \int_\fG \chi(x) \varphi_1(x-s) d  x. 
\end{align*} 
Here $\int_\fG \chi(x)\varphi_1(x-s)d x=\chi(s)\widehat
 {\varphi_1}(-\chi)$ is
 $\cconj{\chi^\star(s^\star)}\widehat{\varphi_1}(-\chi)$ by
 \eqref{juji}. Therefore
\begin{align}
\dual{(\iota\Phi_1)} (\xi) =
 \widehat{\varphi_1}(-\chi) \lim_n\frac{1}{|D_n|} \sum_{s\in \Lambda(W)\cap
 D_n} b(s^\star) \cconj{\chi^\star(s^\star)}.\label{eq21}
\end{align}
 Since $b\in\cbp$, the Fourier transform  $\hat b$ of $b$ is
 well-defined. So Proposition~\ref{prop:weyl} implies that the
 limit in \eqref{eq21} is $\int_W b(y)\cconj{\chi^\star}(y)
 d y/|\tilde L|$, which is $\hat b( \chi^\star) / |\tilde L|$
 by $\supp b \subset W$. Therefore \eqref{eq21} implies the desired
 consequence. 
\qed\end{proof}

\begin{potthm} By Proposition~\ref{prop:al}, it
 is sufficient to verify
\begin{align}\int_{\fG}(\varphi * \widetilde\varphi )(x)
 d{\gamma_\omega}(x) = \int_{\dual\fG} \left|
 \dual{\varphi}(-\gamma)\right|^2 d\left( \sum_{\chi\in\cL} \frac{\left|
\hat{b}( \chi^\star) \right|^2}{|\tilde L|^2}\delta_\chi\right)(\gamma)\ge0.
\label{eq:b}
\end{align} Since $\widetilde{\gamma_\omega}=\gamma_\omega$, the leftmost integral is $((\varphi * \tilde{\varphi})*
\gamma_\omega)(0)$, which is
 $\inner{\Phi}{\Phi}_\nu$ by \eqref{eq10}. Because each $\Phi$
 corresponds uniquely to $\varphi$ by Lemma~\ref{lem:abu} and because of Lemma~\ref{thm:a}, it is
 $\sum_{\chi\in\cL} {\left|
\widehat{\varphi}(-\chi) {\widehat{b}(\chi^\star)}\right|^2}\cdot |\tilde
 L|^{-2}$,
 the right-hand side of \eqref{eq:b}.\end{potthm}

By the Theorem we have proved, we can see that
the pure-point diffraction is still
observed as long as the sample path of random
weights is continuous on the
internal space of the model set. The condition is comparable Baake-Moody's
sufficient condition for deterministic model sets to have pure-point
diffraction; their condition demands the continuity with respect to the
internal space.

\begin{theorem}
Suppose the same assumptions as in
Theorem~\ref{thm:2}. If a complex-valued stochastic process $\{
 B_y(\omega)\}_{y\in W}$ is such that the sample path
 $b_\omega(y)=1_W(y)B_y(\omega)$ is continuous on $W$ almost every $\omega$,
then
a Dirac comb $\pi(\omega)=\sum_{s\in \Lambda(W)}
 B_{s^\star}(\omega)\delta_s$ has a
 diffraction measure $\widehat{\gamma_\pi}(\omega)$ which expectation 
is  pure-point 
\begin{align*}\mathrm{E}_\omega[\widehat{\gamma_\pi}(\omega)]=\sum_{\chi\in\cL}
\frac{\E{|{\widehat {b_\omega}}(-\chi)|^2}}{|\tilde
 L|^2}\delta_\chi.
\end{align*}\end{theorem}
  For example, if $W=[0,1]$, $\{B_y(\omega)\}_{y\in W}$ is
an Ornstein-Uhlenbeck process.

In \cite{MR2387563}, for particle gases over \textsc{flc} sets with Gibbs random
field under a suitable interaction potential restrictions, Baake-Zint
proved that the diffraction measures do not have singular continuous
components and explicitly described the pure-point component and the
absolutely continuous component by using the covariance of the random
field.

Weak dependence of a random field over a point set $S$ is studied typically with
\emph{Dobrushin interdependence matrix}~\cite[p.~32]{MR1997114}
$(C_{x,y})_{x,y\in S}$. For
example, K\"ulske~\cite{MR1997114} derived that if a Gibbs field has a
Dobrushin interdependence matrix with
the largest row-sum and the largest
column-sum both less than 1, then it satisfies a concentration inequality for
functions. We see that every finitely randomized model set has a Dobrushin interdependence
matrix $(C_{x,y})_{x,y\in S}$ with each row and each column having
bounded number of nonzero entries, 
because a mass at a point $x$ is
independent from that of any point $y$ whenever $x-y$ is not in the
dependency set $D$.
So, it is natural to generalize finitely randomized model sets $S$ to 
 randomized model sets $S$ having
a finite set $D$ with
$\sup_{x\in S} \sum_{y\in S\setminus (D+x)} C_{x,y} <1$, for example. We hope such
randomized model sets with weak dependence satisfy
our
results and a concentration inequality for functions.

\begin{acknowledgements}
The first author thanks Prof. Michael Baake and anonymous referee.
\end{acknowledgements}


\appendix
\section{Proof of ``If $Int(W)\ne\emptyset$ then $\Lambda(W)$ is relatively dense.''}

Variants of this assertion have already appeared in Schlottmann, Lagarias and so on.

Since $\tilde L$ is a lattice of $\fG\times\Gint$, there is a
 complete system $C$ of representatives of a compact, quotient set $(\fG\times\Gint)/\tilde
 L$ such that $C$ is relatively compact. Then $\fG\times \Gint =
 \tilde L + C$. Since the window $W$ has nonempty interior and $L^\star$
 is dense in $\Gint$, it follows that
 $\Gint=L^\star-W$. Moreover, since $\Pi_\rint(Cl(C))$ is compact, there
 is a finite subset $F$ of $L$ such that $\Pi_\rint(Cl(C))\subset
 F^\star-W$. Let $K:=\Pi(Cl(C))-F$, which is compact in $\fG$. For each
 $x\in \fG$, there are $(t,t^\star)\in\tilde L$ and $(c,d)\in C$ such
 that $(x,0)=(t,t^\star)+(c,d)$. Since $\Pi_\rint(C)\subset F^\star-W$,
 there are $f\in F$ and $w\in W$ such that $d=f^\star-w$. Then
\begin{align*}
 (x,0)=(t,t^\star)+(c,f^\star -w) = (t+c,\, t^\star +f^\star -w).
\end{align*}
So $t^\star+f^\star=w\in W$, and thus $t+f\in \Lambda(W)$. Since
 $x=t+c=(t+f)+(c-f)$ and $c-f\in \Pi(C)-F\subset K$, we have $x\in
 \Lambda(W)+K$. Therefore $\fG=\Lambda(W)+K$.
\end{document}